\documentclass[ECP,preprint]{ejpecp}

\usepackage[T1]{fontenc}
\usepackage[utf8]{inputenc}

\usepackage{enumerate}
\usepackage{physics}

\SHORTTITLE{Wasserstein Contraction of a Geodesic Random Walk on the Sphere}
\TITLE{A Dimension-Independent Bound on the Wasserstein Contraction Rate of a Geodesic Random Walk on the Sphere}
\AUTHORS{Philip Schär\footnote{Friedrich Schiller University Jena, Germany. \EMAIL{philip.schaer@uni-jena.de}}, Thilo D. Stier\footnote{Georg August University of Göttingen, Germany. \EMAIL{thilodaniel.stier@uni-goettingen.de}}}

\AMSSUBJ{60G50; 60J05; 65C05}
\SUBMITTED{September 16, 2023}
\ACCEPTED{(...)}

\VOLUME{0}
\YEAR{2020}
\PAPERNUM{0}
\DOI{10.1214/YY-TN}

\ABSTRACT{We theoretically analyze the properties of a geodesic random walk on the Euclidean ${d\text{-sphere}}$. Specifically, we prove that the random walk's transition kernel is Wasserstein contractive with a contraction rate which can be bounded from above independently of the dimension $d$. Our result is of particular interest due to its implications regarding the potential for dimension-independent performance of both geodesic slice sampling on the sphere and Gibbsian polar slice sampling, which are Markov chain Monte Carlo methods for approximate sampling from essentially arbitrary distributions on their respective state spaces.}

\newcommand{\ra}{\rightarrow}
\newcommand{\N}{\mathbb{N}}
\newcommand{\R}{\mathbb{R}}
\newcommand{\sph}{\mathbb{S}}
\newcommand{\sphsurf}{\mathfrak{s}}
\renewcommand{\d}{\text{d}}
\newcommand{\ind}{\mathds{1}}
\newcommand{\B}{\mathcal{B}}
\newcommand{\E}{\mathbb{E}}
\newcommand{\U}{\mathcal{U}}
\newcommand{\gap}{\textsf{gap}}
\newcommand{\W}{\mathcal{W}_{1\!}}
\newcommand{\Dob}{\textsf{Dob}}

\newcommand\ooint[2]{\,]#1,#2[}

\newcommand\coint[2]{\,[#1,#2[}

\begin{document}

\section{Introduction} \label{Sec:intro}

In this work we investigate quantitative theoretical properties of a geodesic random walk on the Euclidean $d$-sphere
\begin{equation*}
	\sph^d
	:= \{x \in \R^{d+1} \mid \norm{x} = 1\}
\end{equation*}
(where $\norm{\cdot}$ denotes the Euclidean norm) for arbitrary dimensions $d \geq 2$. Intuitively, the random walk we consider performs an iteration from its current state $X_{n-1} = x_{n-1}$ to a new state $X_n = x_n$ by two simple steps. First, it selects a closed geodesic $g_n$ of $\sph^d$ running through $x_{n-1}$, uniformly at random among all such geodesics. Then it chooses the new state $x_n$ as a realization of the uniform distribution over all points on $g_n$.

Our contribution is a novel theoretical result on this geodesic random walk, which establishes bounds on the speed of its convergence (in Wasserstein distance) from an arbitrary initial distribution to its stationary distribution (which is the uniform distribution over the sphere). In particular, we prove the dimension-independent upper bound $\rho = 7/8$ on the Wasserstein contraction rate of the walk's transition kernel (see Corollary \ref{Cor:dim_indep_bound} below).

Let us begin by providing a formal description of the walk as a Markov chain on $(\sph^d,\B(\sph^d))$, where $\B(\sph^d)$ denotes the Borel $\sigma$-algebra of $\sph^d$. For this we require some preliminaries. Let $\sigma_d$ be the natural surface measure on $(\sph^d,\B(\sph^d))$ and $\sphsurf_d := \sigma_d(\sph^d)$ its total mass. For any $x \in \sph^d$ let $\sph^{d-1}_x \subseteq \sph^d$ denote the \textit{great subsphere w.r.t.~$x$}, i.e.~the set
\begin{equation*}
	\sph^{d-1}_x := \{z \in \sph^d \mid x^{\top} z = 0\} ,
\end{equation*}
and $\sigma_{d-1}^{(x)}$ the natural surface measure on $\sph^{d-1}_x$. It is well-known that $\sigma_{d-1}^{(x)}(\sph^{d-1}_x) = \sphsurf_{d-1}$ (see e.g.~\cite[Section 2.2]{GSSS}).
Moreover, because $\sph^{d-1}_x$ canonically corresponds to the unit sphere in the tangent space to $\sph^d$ at $x$, it parametrizes the geodesics running through~$x$. Specifically, for any $z \in \sph^{d-1}_x$, a \textit{closed geodesic} (which on $\sph^d$ is the same as a great circle) through $x$ is given by the unit velocity curve
\begin{equation}
	g_{x,z} \colon \coint{0}{2\pi} \, \ra \sph^d, \; \omega \mapsto \cos(\omega) x + \sin(\omega) z .
	\label{Eq:def_g_xz}
\end{equation}
For notational convenience, we denote by $G_{x,z} := g_{x,z}(\!\coint{0}{2\pi}) \subset \sph^d$ the image of $g_{x,z}$.

Drawing a reference point $z \in \sph^{d-1}_x$ from $\U(\sph^{d-1}_x) := \frac{1}{\sphsurf_{d-1}} \sigma_{d-1}^{(x)}$ to determine a geodesic $g_{x,z}$ which runs through $x$ amounts to a uniform distribution over all such geodesics. Furthermore, due to the constant velocity of $g_{x,z}$, drawing $\Omega \sim \U(\!\coint{0}{2\pi})$ and applying $g_{x,z}$ to the result implements a uniform distribution over $G_{x,z}$. We refer to \cite[Section 2]{GSSS} and the references therein for more details on all of these preliminaries.

We may now express the geodesic random walk we described earlier by its transition kernel\footnote{In our nomenclature, a transition kernel is simply a Markov kernel whose arguments belong to the same measurable space. Basic knowledge of Markov kernels on the reader's part is presumed throughout this work. We refer to \cite[Section 1.2]{MCbook} for a sufficiently broad introduction to this topic.} $P$ given by
\begin{equation}
	P(x, A)
	:= \frac{1}{2 \pi \, \sphsurf_{d-1}} \int_{\sph^{d-1}_x} \int_0^{2 \pi} \ind_A(g_{x,z}(\omega)) \d\;\!\omega \, \sigma_{d-1}^{(x)}(\d z) ,
	\label{Eq:GRW}
\end{equation}
for $x \in \sph^d$, $A \in \B(\sph^d)$. 

The remainder of this paper is structured as follows. In Section \ref{Sec:motivation}, we discuss our motivation for examining the geodesic random walk specified above. Then, in Section \ref{Sec:BG}, we provide some background on Wasserstein contraction, the theoretical framework employed in our analysis. Afterwards, in Section \ref{Sec:main_res}, we state and prove our main result. We conclude with a short discussion of that result in Section \ref{Sec:discussion}.

\section{Motivation} \label{Sec:motivation}

Our examination of the geodesic random walk defined by the transition kernel~\eqref{Eq:GRW} is motivated in several different ways. First of all, the convergence rates of geodesic random walks may be of interest as standalone results, and other authors have previously established results similar to ours on such grounds. In this regard, we wish to highlight \cite{Mangoubi}, in which the authors examined a geodesic random walk on a fairly general Riemannian manifold $\mathcal{M}$, which could be chosen as $\sph^d$. Like the present work, the main result of \cite{Mangoubi} relied on the Wasserstein contraction framework to bound the walk's convergence rate by a term which, in the case of the underlying manifold being $\sph^d$, does not depend on the dimension $d$. On the other hand, the walk they considered for $\mathcal{M} = \sph^d$ is (in some sense) simpler than the one we are interested in: It also chooses a random geodesic through the current state (exactly like our walk), but then picks the next state deterministically by taking a step of a fixed size (a hyperparameter) along the selected geodesic (whereas our walk would perform a second sampling step here). Although the aforementioned main result of \cite{Mangoubi} is a much broader one than ours, in the sense that it applies to walks on many different manifolds, whereas we only permit $\sph^d$, our result has the advantage that the walk it considers is also of relevance to a number of methods for (approximate) sampling from essentially arbitrary distributions on either $\sph^d$ or $\R^d$. Before we elaborate on this point, let us emphasize that the proof of our result, despite proceeding within the same theoretical framework as \cite{Mangoubi}, relies on a markedly different strategy, primarily because our restriction to the $d$-sphere allows us to make extensive use of its unique symmetry properties.

As mentioned above, our motivation for analyzing the transition kernel~\eqref{Eq:GRW} also stems from its relevance to the properties of certain instances of slice sampling, which is a general framework for setting up Markov chain Monte Carlo (MCMC) methods for approximately sampling from arbitrary target distributions given by their (potentially unnormalized\footnote{As slice sampling approaches generally do not require the target density to be normalized, we will omit the ``potentially unnormalized'' in the following.}) densities. Specifically, our motivation relates to two particular (hybrid) slice samplers, one based on geodesics and one working in polar coordinates. For each of them, we now explain the sampler's inner workings in some detail and then point out the connection between its properties and this work's main result.

\textit{Geodesic slice sampling on the sphere} (GSSS) \cite{GSSS} is an MCMC method for approximate sampling from target distributions $\nu$ on $(\sph^d, \B(\sph^d))$ (for any $d \geq 1$) given by a density $\varrho \colon \sph^d \ra \coint{0}{\infty}$ relative to $\sigma_d$, meaning that one's only access to $\nu$ is through evaluations of $\varrho$ and one has
\begin{equation*}
	\nu(A)
	= \frac{\int_A \varrho(x) \sigma_d(\d x)}{\int_{\sph^d} \varrho(x) \sigma_d(\d x)} ,
	\qquad A \in \B(\sph^d) .
\end{equation*}
In \cite{GSSS}, the authors of GSSS proposed two different variants of the algorithm, \textit{ideal geodesic slice sampling} and \textit{geodesic shrinkage slice sampling}. They perform a transition from the current state $X_{n-1} = x_{n-1}$ to a new state $X_n = x_n$ by the following steps:
\begin{enumerate}
	\item Draw a threshold $T_n \sim \U(\!\ooint{0}{\varrho(x_{n-1})})$, call the result $t_n$. This defines the (super-) level set or \textit{slice}
	\begin{equation*}
		L(\varrho, t_n)
		:= \{x \in \sph^d \mid \varrho(x) > t_n\} ,
	\end{equation*}
	from which the methods then need to draw the new state $X_n$.
	\item Select a random geodesic $g_{x_{n-1},z_n}$ through $x_{n-1}$ by drawing $Z_n \sim \U(\sph^{d-1}_{x_{n-1}})$ and calling the result $z_n$.
	\item In case of ideal geodesic slice sampling, draw
	\begin{equation}
		\Omega_n \sim \U(\{\omega \in \coint{0}{2 \pi} \, \mid g_{x_{n-1},z_n}(\omega) \in L(\varrho, t_n)\}) , 
		\label{Eq:ideal_GSS_ome_up}
	\end{equation}
	call the result $\omega_n$ and set $x_n := g_{x_{n-1},z_n}(\omega_n)$. In case of geodesic shrinkage slice sampling, draw $\Omega_n \sim U_{\text{geo}}(0, \cdot)$, where $U_{\text{geo}}$ is an auxiliary transition kernel on $\coint{-2\pi}{2\pi} \times \B(\!\coint{-2\pi}{2\pi})$ which mimics sampling from the distribution in \eqref{Eq:ideal_GSS_ome_up} (up to shifts by $2\pi$) by means of a shrinkage procedure (a sort of adaptive acceptance/rejection method). Again call the result $\omega_n$ and set $x_n := g_{x_{n-1},z_n}(\omega_n)$.
\end{enumerate}
In the specific case of the target density $\varrho$ being constant (i.e.~$\varrho(x) = c > 0$, $x \in \sph^d$), the transition kernel of each of the two GSSS variants is easily seen to coincide with \eqref{Eq:GRW}, and so both variants simplify to the geodesic random walk we are interested in. Naturally, $\varrho$ being constant corresponds to the target distribution $\nu$ being the uniform distribution over the entire sphere, $\nu = \U(\sph^d)$. While this case was already covered by \cite[Theorems 15, 26]{GSSS}, those theorems only yielded a dimension-dependent convergence rate bound\footnote{In the two referenced results, this $\rho(d)$ plays the role that the total variation distance between the distribution of the $n$-th Markov chain iterate $X_n$ and the target distribution $\nu$ is upper-bounded by $\rho(d)^n$.} $\rho(d)$ which deteriorates with increasing dimension, in the sense that $\lim_{d \ra \infty} \rho(d) = 1$. In contrast, our result proves that GSSS for the uniform target distribution $\nu = \U(\sph^d)$ works well in any dimension $d \geq 2$.

As one may be inclined to disregard this result on the grounds that exact sampling from $\U(\sph^d)$ is computationally tractable (so that there is no point in applying MCMC to sample from this distribution in the first place), we should point out that the Wasserstein contraction property we establish is known to be robust under small perturbations of the transition kernel\footnote{More concretely, a sufficiently small perturbation of a Wasserstein contractive kernel is still Wasserstein contractive, albeit likely with a slightly worse contraction rate.}, cf.~\cite[Remark 54]{Ollivier}. Moreover, since the behavior of GSSS is easily seen to be fairly robust under small perturbations of the target distribution $\nu$ (unlike e.g.~gradient-based methods), we expect such changes to the target would only cause small perturbations to its transition kernel, so that the Wasserstein contraction rate from the uniform case $\nu = \U(\sph^d)$ should, to some extent, be preserved. Under this premise, our result on GSSS for the uniform target should intuitively transfer to GSSS for ``near-uniform'' target distributions, i.e.~ones whose densities $\varrho$ are not quite constant, but ``flat'' in the sense that their range $\sup \varrho - \inf \varrho$ is relatively small compared to their minimum value $\inf \varrho$. That is, the true convergence rate of GSSS for such near-uniform targets should be much closer to the dimension-independent rate bound we provide for the case $\nu = \U(\sph^d)$ than to the dimension-dependent rate bounds one obtains from the results in \cite{GSSS} for such near-uniform targets.

This in turn suggests GSSS could be useful as a base sampler for parallel tempering~\cite{Neal_PT} on high-dimensional spheres: In such schemes, one needs to apply the base sampler not only to a given target density $\varrho$ on the sphere, but also to ``heated-up'' versions of $\varrho$, which become flatter and flatter with increasing (simulated) temperature and thereby eventually enter the aforementioned near-uniform regime. Whereas the performances of typical choices of the base sampler, e.g.~the random walk Metropolis algorithm, are known to slowly deteriorate with increasing dimension, GSSS should, according to our above considerations, be able to maintain an excellent performance even in high dimensions (at least for the near-uniform high-temperature versions of the target). We therefore surmise that the use of GSSS as the base sampler in parallel tempering schemes on the sphere has the potential to substantially improve the overall performance of these schemes, compared to the use of a more conventional base sampler.

\textit{Gibbsian polar slice sampling} (GPSS) \cite{GPSS} is an MCMC method that can, in principle, be applied to any target distribution $\zeta$ on $(\R^d,\B(\R^d))$ (for any $d \geq 2$) given by a (potentially unnormalized) Lebesgue density $\eta \colon \R^d \ra \coint{0}{\infty}$, meaning
\begin{equation*}
	\zeta(A)
	= \frac{\int_A \eta(x) \d x}{\int_{\R^d} \eta(x) \d x} ,
	\qquad A \in \B(\R^d) .
\end{equation*}
GPSS replaces a given target density $\eta$ by the transformation $\widetilde{\eta}(x) := \norm{x}^{d-1} \eta(x)$, which allows it to generate its samples in polar coordinates, i.e.~as the product of a $\ooint{0}{\infty}$-valued random variable (the \textit{radius}) and an $\sph^{d-1}$-valued random variable (the \textit{direction}). Specifically, GPSS performs a transition from the current state $X_{n-1} = x_{n-1}$ to a new state $X_n = x_n$ by the following steps:
\begin{enumerate}
	\item Draw a threshold $T_n \sim \U(\!\ooint{0}{\widetilde{\eta}(x_{n-1})})$, call the result $t_n$ and define the slice
	\begin{equation*}
		L(\widetilde{\eta}, t_n)
		:= \{x \in \R^d \mid \widetilde{\eta}(x) > t_n\} ,
	\end{equation*}
	from which the new state $X_n$ must then be drawn.
	\item Let $r_{n-1} := \norm{x_{n-1}} \in \ooint{0}{\infty}$ and $\theta_{n-1} := x_{n-1} / r_{n-1} \in \sph^{d-1}$.
	\item Direction update: Draw $\Theta_n \sim Q_n(\theta_{n-1}, \cdot)$, where $Q_n$ denotes the transition kernel of GSSS for the target distribution
	\begin{equation*}
		\nu_n
		:= \U(\{\theta \in \sph^{d-1} \mid r_{n-1} \theta \in L(\widetilde{\eta},t_n)\}) ,
	\end{equation*}
	given by its density
	\begin{equation*}
		\varrho_n(\theta)
		:= \ind_{L(\widetilde{\eta},t_n)}(r_{n-1} \theta) ,
		\qquad \theta \in \sph^{d-1} .
	\end{equation*}
	Call the result $\theta_n$.
	\item Radius update: Draw $R_n \sim U_n(r_{n-1},\cdot)$, where $U_n$ denotes an auxiliary transition kernel on a subset of $\R$, designed to mimic sampling from
	\begin{equation*}
		\U(\{r \in \ooint{0}{\infty} \; \mid r \theta_n \in L(\widetilde{\eta}, t_n)\}) .
	\end{equation*}
	Call the result $r_n$. Finally, set $x_n := r_n \theta_n$.
\end{enumerate}
Given how GPSS relies on GSSS within its transition mechanism, it is perhaps not too surprising that there are strong links between their respective properties. Here we are particularly interested in the properties of GPSS for target densities $\eta$ which are \textit{rotationally invariant}, meaning $\eta(x) = h_{\eta}(\norm{x})$ for some $h_{\eta} \colon \coint{0}{\infty} \; \ra \coint{0}{\infty}$ and all $x \in \R^d$.
Let $r_{n-1}$ and $t_n$ be as in the above enumeration. Since, by construction,
\begin{equation*}
	\widetilde{\eta}(r_{n-1} \theta)
	= (r_{n-1})^{d-1} h_{\eta}(r_{n-1})
	= \widetilde{\eta}(x_{n-1})
	> t_n ,
	\qquad \theta \in \sph^{d-1} ,
\end{equation*}
one then has $r_{n-1} \theta \in L(\widetilde{\eta},t_n)$ for any $\theta \in \sph^{d-1}$, and thus $\nu_n = \U(\sph^{d-1})$. In other words, GPSS applied to a rotationally invariant target density performs its direction update by means of GSSS applied to $\U(\sph^{d-1})$. Thus our main result in this work not only permits an alternative view as a result on GSSS for the uniform target $\nu = \U(\sph^d)$, but also another one as a result on the direction update of GPSS for rotationally invariant targets.

Moreover, the examination of GPSS for such targets is well motivated by both theoretical considerations and practical observations from previous works. The motivation from the theoretical side arises as follows: In \cite{GPSS}, GPSS is constructed as a computationally efficient approximation to \textit{polar slice sampling} (PSS) \cite{RoRoPSS}, which transitions similarly to GPSS but, instead of drawing $R_n$ and $\Theta_n$ separately, relies on the joint update
\begin{equation*}
	(R_n, \Theta_n) \sim \U(\{(r,\theta) \in \ooint{0}{\infty} \, \times \sph^{d-1} \mid r \theta \in L(\widetilde{\eta},t_n)\}) .
\end{equation*}
The recent works \cite{PSS_gap,k-PSS} showed that PSS performs dimension-independently well (in terms of its spectral gap) on two classes of rotationally invariant target densities, one light-tailed and one heavy-tailed. Naturally, these results raise the question whether or not the dimension-independence of PSS is retained by GPSS. According to the recent work \cite{WPI_HSS}, the key quantity to examine in order to answer this question is the spectral gap of the \textit{on-slice transition kernel} $H_t$, which in our notation is determined by the current threshold $t := t_n$, takes the current state $x_{n-1}$ as input, and outputs the new state $x_n$. Admittedly, for a complete proof that GPSS does indeed have dimension-independent spectral gaps in the same cases as PSS, we would still need to conduct a considerable amount of further analysis, which we leave for future work. Nevertheless, the result we establish here should be an important ingredient for such a proof (besides the results on PSS from \cite{PSS_gap,k-PSS}), particularly because the Wasserstein contraction property we prove is well-known to imply an explicit lower bound on the spectral gap \cite[Proposition 30]{Ollivier}.

From the practical side, GPSS was already empirically observed to exhibit \textit{seemingly} dimension-independent performance for (certain) rotationally invariant targets by its authors in \cite{GPSS}. However, the empirical analysis of an MCMC method's output is generally insufficient for definitively confirming such a conjecture, so to be sure one must undertake a theoretical examination instead.

A strong additional motivation for analyzing GPSS for rotationally invariant targets stems from the recent work \cite{PATT_acc}. There the authors propose \textit{parallel affine transformation tuning} (PATT), a sampling scheme based on multiple interacting chains which collaboratively learn an affine transformation of the target density, aimed at bringing it into isotropic position (i.e.~mean zero, covariance is identity matrix). Essentially, each of these chains proceeds by running a \textit{base sampler} for the transformed target. While the base sampler may be an arbitrary MCMC method, the authors in \cite{PATT_acc} advocate designating GPSS the default choice for this role. Furthermore, for sufficiently regular target densities, their affine transformations into isotropic position either attain rotational invariance outright or at least come very close to it. Hence the performance of PATT with its default base sampler GPSS, denoted PATT-GPSS in \cite{PATT_acc}, is strongly dependent on the performance of GPSS for targets which are rotationally invariant or close to it. Like for GSSS and uniform targets, we surmise the performance of GPSS for targets which are nearly rotationally invariant should be close to its performance for rotationally invariant targets, but for now it remains an open problem to theoretically confirm this.

\section{Background on Wasserstein Contraction} \label{Sec:BG}

Before we can move on to our main result, we need to briefly introduce the relevant notions concerning Wasserstein contraction, for which we loosely follow the corresponding parts of \cite{k-PSS}. For two probability measures $\xi_1, \xi_2$ on $(\sph^d,\B(\sph^d))$, denote by $\Gamma(\xi_1,\xi_2)$ the set of couplings between them, i.e.~of probability measures $\gamma$ on $(\sph^d \times \sph^d, \B(\sph^d \times \sph^d))$ that satisfy
\begin{align*}
	&\gamma(A \times \sph^d) = \xi_1(A) , \\
	&\gamma(\sph^d \times A) = \xi_2(A)
\end{align*}
for any $A \in \B(\sph^d)$.
The restriction of the Euclidean distance on $\R^{d+1}$ to $\sph^d \subseteq \R^{d+1}$ is easily seen to constitute a metric on $\sph^d$. We can therefore define the Wasserstein distance (of order $1$) between $\xi_1$ and $\xi_2$ w.r.t.~the Euclidean distance,
\begin{equation*}
	\W(\xi_1,\xi_2)
	:= \inf_{\gamma \in \Gamma(\xi_1,\xi_2)} \int_{\sph^d \times \sph^d} \norm{v_1 - v_2} \gamma(\d v_1 \times \d v_2) .
\end{equation*}
Let $P$ be a Markov kernel on $\sph^d \times \B(\sph^d)$, then its \textit{Dobrushin coefficient} (w.r.t.~the Euclidean distance) is given by
\begin{equation}
	\Dob(P)
	:= \sup_{x,y \in \sph^d, x \neq y} \frac{\W(P(x,\cdot),P(y,\cdot))}{\norm{x - y}} ,
	\label{Eq:def_Dob}
\end{equation}
see \cite[Definition 20.3.1, Lemma 20.3.2]{MCbook}. We now say that $P$ is \textit{Wasserstein contractive with rate bound $\rho$} if $\Dob(P) \leq \rho < 1$. 
This type of contraction is known to be an extremely powerful property: If $P$ is Wasserstein contractive with rate bound $\rho$ and admits $\nu$ as its invariant distribution\footnote{Note that we know the invariant distribution of the transition kernel $P$ defined in \eqref{Eq:GRW} to be $\nu := \U(\sph^d)$, since the kernel coincides with that of GSSS for target distribution $\U(\sph^d)$ (cf.~Section \ref{Sec:motivation}) and GSSS always leaves its target distribution invariant \cite[Lemma 14]{GSSS}.}, then for any distribution $\mu$ on $(\sph^d,\B(\sph^d))$ one has
\begin{equation}
	\W(\mu P^n, \nu) \leq \rho^n \W(\mu, \nu), \qquad n \in \N_0
	\label{Eq:W_conv}
\end{equation}
(see e.g.~\cite[Theorem 20.3.4]{MCbook}), where $\mu P^n$ corresponds to the distribution of $X_n$ when $(X_n)_{n \in \N_0}$ is a Markov chain with initial distribution $\mu$ and transition kernel $P$. Moreover, it is well-known (see e.g.~\cite[Proposition 30]{Ollivier}) that, under the same assumptions, one gets
\begin{equation}
	\gap_{\nu}(P) \geq 1 - \rho ,
	\label{Eq:contraction_to_gap}
\end{equation}
where $\gap_{\nu}(P)$ denotes the ($L_2(\nu)$-)spectral gap of $P$. The spectral gap is known to be a key property of Markov chains, as it quantifies both the convergence speed in terms of the total variation distance and the asymptotic sample quality in terms of Monte Carlo integration errors. For details, we refer to \cite[Sections 1-2]{PSS_gap} and the references therein. Simply put, an explicit upper bound on the \textit{Wasserstein contraction rate} (we use this term synonymously to Dobrushin coefficient) of a kernel $P$ yields not only the geometric convergence \eqref{Eq:W_conv} but also the explicit bound \eqref{Eq:contraction_to_gap} on the spectral gap and by extension the various additional results such a bound entails.

\section{Main Result} \label{Sec:main_res}

Using the concepts presented in Section \ref{Sec:BG}, we can now formulate our main result:

\begin{theorem} \label{Thm:main_res}
	For any fixed $d \geq 2$, the transition kernel $P$ defined in \eqref{Eq:GRW}, which describes the geodesic random walk on $\sph^d$, is Wasserstein contractive with rate bound
	\begin{equation}
		\rho
		:= \frac{1}{2 \pi} \int_0^{2 \pi} \sqrt{\cos^2(\omega) + \frac{\sin^2(\omega)}{d} } \; {\normalfont\d}\;\!\omega .
		\label{Eq:contraction_rate}
	\end{equation}
\end{theorem}

Clearly, this bound is strictly decreasing in $d$. Thus, by computing its value for the smallest permissible dimension, $d = 2$, we obtain a dimension-independent bound on the rate that applies in all dimensions $d \geq 2$. Numerical integration shows that for $d = 2$ the rate bound \eqref{Eq:contraction_rate} is approximately $0.86$. Rounding this up to a nicer value, we end up with the following corollary.

\begin{corollary} \label{Cor:dim_indep_bound}
	For any fixed $d \geq 2$, the transition kernel $P$ defined in \eqref{Eq:GRW} is Wasserstein contractive with rate bound $\rho := 7/8$.
\end{corollary}

We postpone any further discussion of Theorem \ref{Thm:main_res} and Corollary \ref{Cor:dim_indep_bound} to the next section and focus here on proving the theorem. Fix $P$ to be as in \eqref{Eq:GRW} for the remainder of the section. Before we delve into any computations, we observe that the inherent symmetry of $\sph^d$ and the random walk kernel \eqref{Eq:GRW} on it allow us to make some important simplifications to the definition \eqref{Eq:def_Dob} of $\Dob(P)$ in our particular case. Namely, thanks to the rotational symmetry\footnote{Specifically, due to the symmetry of $\sph^d$ as a Riemannian manifold, the transition probabilities of $P$ are unaffected by rotations around the origin (applied to both its arguments).}, $\W(P(x,\cdot), P(y,\cdot))$ for any $x, y \in \sph^d$ can only depend on the (Euclidean) distance between $x$ and $y$, not on their absolute positions on the sphere.
It is therefore sufficient to arbitrarily fix $x \in \sph^d$ and only consider points $y$ from the image of a geodesic\footnote{A suitably chosen half of a geodesic would also suffice.} $g_{x,z}$ through $x$ from which $x$ itself is removed, i.e.~$G_{x,z} \setminus \{x\}$. As we may freely choose both $x$ and $z$, we choose them as points whose Euclidean coordinates are convenient to our proof, specifically
\begin{align}
	\begin{split}
		&x := (1,0,\dots,0)^{\top} \in \sph^d , \\
		&z := (0,1,0,\dots,0)^{\top} \in \sph^{d-1}_x .
	\end{split}
	\label{Eq:def_x_z}
\end{align}
Note that this choice of $z$ yields
\begin{equation}
	G_{x,z} \setminus \{x\} 
	= \{ y_{\alpha} := (\cos(\alpha), \sin(\alpha), 0,\dots,0)^{\top} \in \sph^d \mid \alpha \in \ooint{0}{2\pi} \, \} .
	\label{Eq:def_S_yalpha}
\end{equation}
Using the definitions in \eqref{Eq:def_x_z}, \eqref{Eq:def_S_yalpha}, we may now rewrite $\Dob(P)$ as
\begin{equation}
	\Dob(P) 
	= \sup_{\alpha \in \ooint{0}{2\pi}} \frac{\W(P(x,\cdot), P(y_{\alpha},\cdot))}{\norm{x - y_{\alpha}}}_{\textbf{.}}
	\label{Eq:Dob_symmetry}
\end{equation}

It remains to show that \eqref{Eq:Dob_symmetry} is upper-bounded by \eqref{Eq:contraction_rate}. For the sake of a simpler proof structure, we first establish the following auxiliary result.

\begin{lemma} \label{Lem:rotation_is_coupling}
	For any $\alpha \in \R$, the rotation matrix
	\begin{equation*}
		R_{\alpha} = \begin{pmatrix} 
			\cos(\alpha) & -\sin(\alpha) & & & \\
			\sin(\alpha) & \cos(\alpha) & & & \\
			& & 1 & & \\
			& & & \ddots & \\
			& & & & 1 \\
		\end{pmatrix} 
		\in \R^{(d+1) \times (d+1)}
	\end{equation*}
	defines a transport map between $P(x,\cdot)$ and $P(y_{\alpha},\cdot)$ via $v \mapsto R_{\alpha} v$. That is, the probability measure $\gamma_{\alpha}$ on $\sph^d \times \sph^d$ determined by
	\begin{equation*}
		\gamma_{\alpha}(A \times B)
		:= \int_{\sph^d} \ind_A(v) \ind_B(R_{\alpha} v) P(x, \d v) ,
		\qquad A, B \in \B(\sph^d)
	\end{equation*}
	is a coupling between $P(x,\cdot)$ and $P(y_{\alpha},\cdot)$.
\end{lemma}
\begin{proof}
	We immediately get
	\begin{equation*}
		\gamma_{\alpha}(A \times \sph^d)
		= \int_{\sph^d} \ind_A(v) P(x, \d v) 
		= P(x, A) ,
		\qquad A \in \B(\sph^d) ,
	\end{equation*}
	which is already half of what we need to show. The second half turns out to be slightly more complicated but still straightforward. For the sake of brevity, we denote by $c_d := (2\pi \sphsurf_{d-1})^{-1}$ the normalization constant of $P$. We then get
	\begin{align*}
		\gamma_{\alpha}(\sph^d \times B)
		&= \int_{\sph^d} \ind_B(R_{\alpha} v) P(x, \d v) 
		= \int_{\sph^d} \ind_{R_{-\alpha}(B)}(v) P(x, \d v) \\ 
		&= P(x, R_{-\alpha}(B)) 
		= c_d \int_{\sph^{d-1}_x} \int_0^{2\pi} \ind_{R_{-\alpha}(B)}(g_{x,z}(\omega)) \d\;\!\omega \, \sigma_{d-1}^{(x)}(\d z) \\ 
		&= c_d \int_{\sph^{d-1}_x} \int_0^{2\pi} \ind_{R_{-\alpha}(B)}(\cos(\omega) x + \sin(\omega) z) \d\;\!\omega \, \sigma_{d-1}^{(x)}(\d z) \\ 
		&= c_d \int_{\sph^{d-1}_x} \int_0^{2\pi} \ind_{B}(\cos(\omega) R_{\alpha} x + \sin(\omega) R_{\alpha} z) \d\;\!\omega \, \sigma_{d-1}^{(x)}(\d z) \\ 
		&= c_d \int_{\sph^{d-1}_{y_{\alpha}}} \int_0^{2\pi} \ind_{B}(\cos(\omega) y_{\alpha} + \sin(\omega) z) \d\;\!\omega \, \big( (R_{\alpha})_{\#} \sigma_{d-1}^{(x)} \big) \! (\d z) \\ 
		&= c_d \int_{\sph^{d-1}_{y_{\alpha}}} \int_0^{2\pi} \ind_{B}(\cos(\omega) y_{\alpha} + \sin(\omega) z) \d\;\!\omega \, \sigma_{d-1}^{(y_{\alpha})} (\d z) \\ 
		&= P(y_{\alpha}, B)
	\end{align*}
	for any $B \in \B(\sph^d)$. Note that we make repeated use of $R_{\alpha} x = y_{\alpha}$ (which trivially follows from their respective definitions) and that the second to last equation relies on the fact that $R_{\alpha}$ is an isometry from $\sph^{d-1}_x$ to $\sph^{d-1}_{y_{\alpha}}$.
\end{proof}

With the identity \eqref{Eq:Dob_symmetry} and Lemma \ref{Lem:rotation_is_coupling} at hand, we are now well-equipped to prove the theorem.

\begin{proof}[Proof of Theorem \ref{Thm:main_res}]
	Fix $\alpha \in \ooint{0}{2\pi}$ for now. Observe that for $v = (v_1,\dots,v_{d+1})^{\top} \in \sph^d$, one has
	\begin{equation*}
		R_{\alpha} v = (\cos(\alpha) v_1 - \sin(\alpha) v_2, ~ \sin(\alpha) v_1 + \cos(\alpha) v_2, ~ v_3,\dots,v_{d+1})^{\top} ,
	\end{equation*}
	so that
	\begin{align*}
		\norm{v - R_{\alpha} v}^2
		&= \big((1-\cos(\alpha)) v_1 + \sin(\alpha) v_2 \big)^2 + \big((1-\cos(\alpha)) v_2 - \sin(\alpha) v_1 \big)^2 \\ 
		&= (1-\cos(\alpha))^2 v_1^2 + 2 (1-\cos(\alpha)) \sin(\alpha) v_1 v_2 + \sin^2(\alpha)  v_2^2 \\
		&\quad + (1-\cos(\alpha))^2 v_2^2 - 2 (1-\cos(\alpha)) \sin(\alpha) v_1 v_2 + \sin^2(\alpha) v_1^2 \\ 
		&= \big( (1-\cos(\alpha))^2 + \sin^2(\alpha) \big) (v_1^2 + v_2^2) \\ 
		&= \big( 1 - 2 \cos(\alpha) + \cos^2(\alpha) + \sin^2(\alpha) \big) (v_1^2 + v_2^2) \\ 
		&= 2 (1 - \cos(\alpha)) (v_1^2 + v_2^2) . 
	\end{align*}
	Moreover, since $R_{\alpha} x = y_{\alpha}$, plugging $v := x$ (still defined as in \eqref{Eq:def_x_z}) into the above computation yields
	\begin{equation*}
		\norm{x - y_{\alpha}}^2 
		= 2 (1 - \cos(\alpha)) .
	\end{equation*}
	Applying Lemma \ref{Lem:rotation_is_coupling} and then plugging in the above, we obtain
	\begin{align*}
		\frac{\W(P(x,\cdot), P(y_{\alpha},\cdot))}{\norm{x - y_{\alpha}}}
		&\leq \frac{1}{{\norm{x - y_{\alpha}}}} \int_{\sph^d \times \sph^d} \norm{v - \tilde{v}} \gamma_{\alpha}(\d v \times \d \tilde{v}) \\ 
		&= \frac{1}{{\norm{x - y_{\alpha}}}} \int_{\sph^d} \norm{v - R_{\alpha} v} P(x, \d v) \\ 
		&= \int_{\sph^d} \sqrt{v_1^2 + v_2^2} \, P(x, \d v) . 
	\end{align*}
	Remarkably, the final expression above does not depend on $\alpha$ anymore. Therefore, combining this result with \eqref{Eq:Dob_symmetry}, we immediately get
	\begin{equation}
		\Dob(P) 
		\leq \int_{\sph^d} \sqrt{v_1^2 + v_2^2} \, P(x, \d v) .
		\label{Eq:Dob_without_sup}
	\end{equation}
	
	In the remainder of the proof, we further simplify the above integral. To increase legibility, we rely on the following notation: For a probability space $(H,\mathcal{H},\xi)$ and a measurable function $f \colon (H,\mathcal{H}) \ra (\R,\B(\R))$, we denote by
	\begin{equation*}
		\E_{W \sim \xi}(f(W))
		:= \int_H f(w) \xi(\d w)
	\end{equation*}
	the expectation of the transformation $f(W)$ of a random variable $W \sim \xi$.
	
	By definition, points in $\sph^{d-1}_x$ are perpendicular to $x$, so that
	\begin{equation}
		u_1 = x^{\top} u = 0 \qquad \forall u = (u_1,\dots,u_{d+1})^{\top} \in \sph^{d-1}_x .
		\label{Eq:u_1_is_zero}
	\end{equation}
	Moreover, $\U(\sph^{d-1}_x)$ is invariant under permutations of the vector components at indices $2$ to $d+1$, so that we must have
	\begin{equation*}
		1 
		= \E_{U \sim \, \U(\sph^{d-1}_x)}(\norm{U}^2) 
		\stackrel{\eqref{Eq:u_1_is_zero}}{=} \sum_{i=2}^{d+1} \E_{U \sim \U(\sph^{d-1}_x)}(U_i^2) 
		= d \cdot \E_{U \sim \U(\sph^{d-1}_x)}(U_2^2) , 
	\end{equation*}
	meaning
	\begin{equation}
		\E_{U \sim \U(\sph^{d-1}_x)}(U_2^2)
		= \frac{1}{d}_{\textbf{.}}
		\label{Eq:expectation_subsphere}
	\end{equation}
	\begin{minipage}{\linewidth} 
	Finally, we bound $\Dob(P)$ by applying, in this order (one bullet point per (in)equality),
	\begin{itemize}
		\item inequality \eqref{Eq:Dob_without_sup}
		\item the definitions \eqref{Eq:GRW}, \eqref{Eq:def_g_xz} of $P$ and $g_{x,u}$ 
		\item the definition \eqref{Eq:def_x_z} of $x$ and equation \eqref{Eq:u_1_is_zero}
		\item Fubini's theorem and the reformulation of an integral as an expectation
		\item the fact that $r \mapsto \sqrt{\cos^2(\omega) + \sin^2(\omega) r}$ is a concave function on $\coint{0}{\infty}$ (for any $\omega$), Jensen's inequality and monotonicity of the integral
		\item equation \eqref{Eq:expectation_subsphere}
	\end{itemize}
	\end{minipage}
	and obtain
	\begin{align*}
		\Dob(P)
		&\leq \int_{\sph^d} \sqrt{v_1^2 + v_2^2} \, P(x, \d v) \\
		&= \frac{1}{2 \pi \, \sphsurf_{d-1}} \int_{\sph^{d-1}_x} \int_0^{2\pi} \sqrt{\textstyle \sum_{i=1}^2 \big(\! \cos(\omega) x_i + \sin(\omega) u_i \big)^2} \; \d\;\!\omega \, \sigma_{d-1}^{(x)}(\d u) \\ 
		&= \frac{1}{2 \pi \, \sphsurf_{d-1}} \int_{\sph^{d-1}_x} \int_0^{2\pi} \sqrt{\cos^2(\omega) + \sin^2(\omega) u_2^2} \; \d\;\!\omega \, \sigma_{d-1}^{(x)}(\d u) \\ 
		&= \frac{1}{2 \pi} \int_0^{2\pi} \E_{U \sim \, \U(\sph^{d-1}_x)}\!\left(\! \sqrt{\cos^2(\omega) + \sin^2(\omega) U_2^2} \right) \d\;\!\omega \\
		&\leq \frac{1}{2 \pi} \int_0^{2\pi} \sqrt{\cos^2(\omega) + \sin^2(\omega) \E_{U \sim \, \U(\sph^{d-1}_x)}(U_2^2)} \; \d\;\!\omega \\
		&= \frac{1}{2 \pi} \int_0^{2\pi} \sqrt{\cos^2(\omega) + \frac{\sin^2(\omega)}{d}} \, \d\;\!\omega .
		\qedhere 
	\end{align*}
\end{proof}

\section{Discussion} \label{Sec:discussion}

Let us conclude the paper with some brief comments on two aspects of our main result, Theorem \ref{Thm:main_res}. Firstly, in the limit $d \ra \infty$, the rate bound \eqref{Eq:contraction_rate} simplifies to
\begin{equation*}
	\rho 
	= \frac{1}{2 \pi} \int_0^{2 \pi} \abs{\cos(\omega)} \d \omega
	= \frac{2}{\pi}
	\approx 0.64 .
\end{equation*}
As can be seen in Figure \ref{Fig:rho_conv}, the value of \eqref{Eq:contraction_rate} takes quite some time to converge to this asymptotic rate bound, only becoming ``visually indistinguishable'' from it when the dimension is in the order of hundreds or thousands.

Secondly, the rate bound \eqref{Eq:contraction_rate} being strictly decreasing means that Theorem \ref{Thm:main_res} actually yields \textit{stronger} guarantees as the dimension $d$ increases. This suggests that we may be dealing with a case of the blessing of dimensionality, where increases in the dimension lead to improvements in a method's performance. However, we believe this effect to merely be a peculiar consequence of our proof strategy, and conjecture that the true contraction rate of the random walk (i.e.~its Dobrushin coefficient) is in fact either constant in the dimension $d$ or increasing in it.

\section*{Acknowledgements}

We are thankful for the helpful feedback of an anonymous referee.
We thank Björn Sprungk for the discussion that led to our investigation, in particular for drawing our attention to the question of dimension dependence of GSSS for a constant target, in the context of its relevance for GPSS. We thank Mareike Hasenpflug for discussions on the topic and helpful comments on a preliminary version of the manuscript.
PS gratefully acknowledges funding by the Carl Zeiss Stiftung within the program ``CZS Stiftungsprofessuren'' and the project ``Interactive Inference''. Moreover, PS is thankful for support by the DFG within project 432680300 -- Collaborative Research Center 1456 ``Mathematics of Experiment''.

\newpage 
\begin{figure}[t]
	\centering 
	\includegraphics[width=0.6\linewidth]{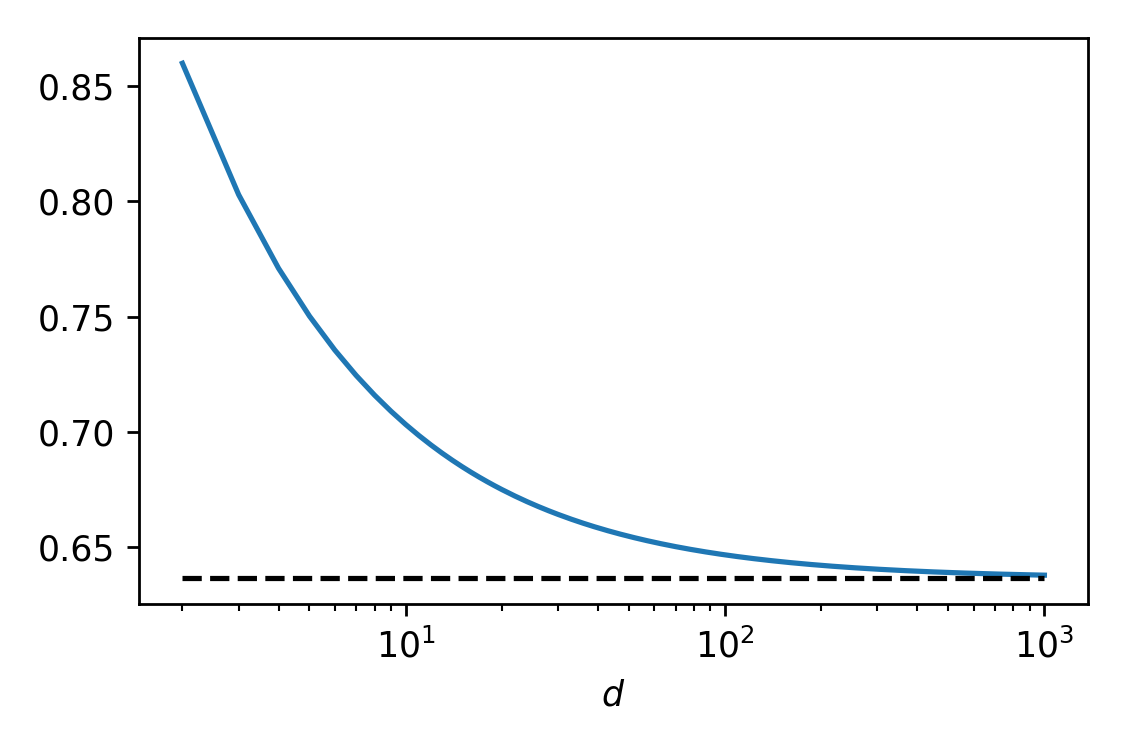}
	\caption{Convergence of the contraction rate bound \eqref{Eq:contraction_rate} to its asymptotic value. The blue line represents the value of \eqref{Eq:contraction_rate} at each integer dimension $d=2,3,\dots,10^3$. The dashed black line marks the asymptotic rate bound $2 / \pi$.}
	\label{Fig:rho_conv}
\end{figure}


\providecommand{\bysame}{\leavevmode\hbox to3em{\hrulefill}\thinspace}
\providecommand{\MR}{\relax\ifhmode\unskip\space\fi MR }
\providecommand{\MRhref}[2]{%
	\href{http://www.ams.org/mathscinet-getitem?mr=#1}{#2}
}
\providecommand{\href}[2]{#2}

\end{document}